\documentclass[12pt]{amsart}  

\usepackage{amsmath,amsthm,amssymb,amscd,mathdots,enumerate,shuffle,hyperref}
\usepackage{setspace,mathrsfs}

\usepackage{tikz}
\usetikzlibrary{decorations.pathreplacing,angles,quotes}
\usetikzlibrary{calc}
\usetikzlibrary{arrows}

\allowdisplaybreaks

\newtheorem{theorem}{Theorem}[section]

\newtheorem{definition}[theorem]{Definition}
\newtheorem{lemma}[theorem]{Lemma}

\newtheorem{ex}[theorem]{Example}
\newtheorem{remark}[theorem]{Remark}

\numberwithin{equation}{section}

\newcommand{\h}{\mathfrak H}
\newcommand{\F}{ \mathcal{H}}

\newcommand{\Q}{\mathbb{Q}}
\newcommand{\R}{\mathbb{R}}
\newcommand{\I}{\mathbb{I}}
\newcommand{\id}{\operatorname{id}}

\title{Rooted tree maps and the derivation relation for multiple zeta values}

\author{Henrik Bachmann}
\address{Graduate School of Mathematics,  Nagoya University, Nagoya, Japan.}
\email{henrik.bachmann@math.nagoya-u.ac.jp}

\author{Tatsushi Tanaka}
\address{Department of Mathematics, Kyoto Sangyo University, Kyoto, Japan.}
\email{t.tanaka@cc.kyoto-su.ac.jp}


\begin{document}
\maketitle

\begin{abstract} 
Rooted tree maps assign to an element of the Connes-Kreimer Hopf algebra of rooted trees a linear map on the noncommutative polynomial algebra in two letters. Evaluated at any admissible word these maps induce linear relations between multiple zeta values. In this note we show that the derivation relations for multiple zeta values are contained in this class of linear relations.
\end{abstract}

\section{Introduction}

In \cite{T} the second named author introduced rooted tree maps, which assign to a rooted forest a linear map on the space $\h = \Q\langle x,y \rangle$ of words in $x$ and $y$. One application of these maps is that any admissible word evaluated at a rooted forest map gives a $\Q$-linear relation between multiple zeta values. In \cite{IKZ} the authors introduced a derivation $\partial_n$ on $\h$ (with respect to the concatenation product), which also gives linear relation between multiple zeta values when evaluated at an admissible word. 

The purpose of this note is to show, that the derivation $\partial_n$ can be written as linear combination of rooted tree maps. In particular the derivation relation of multiple zeta values is a special case of the linear relations of multiple zeta values obtained in \cite{T}.

A rooted tree is as a finite graph which is connected, has no
cycles, and has a distinguished vertex called the root. We draw rooted trees with the root on top and we just consider rooted trees with no plane structure, which means that we for example do not distinguish between  $\begin{tikzpicture}[scale=0.25,baseline={([yshift=-.5ex]current bounding box.center)}]
\def\cz{5}
\def\wi{0.5}

\newcommand{\ci}[1]{	
	\fill[black] (#1) circle (\cz pt);
	\draw (#1) circle (\cz pt);
}

\coordinate (R) at (0,0);

\coordinate (r1) at (\wi,-1);
\coordinate (l1) at (-\wi,-1);

\coordinate (l11) at (-\wi-\wi,-2);
\coordinate (l12) at (-\wi+\wi,-2);

\draw (R) to (r1);
\draw (R) to (l1);

\draw (l1) to (l11);
\draw (l1) to (l12);

\ci{R}
\ci{r1}
\ci{l1}
\ci{l11}
\ci{l12}
\end{tikzpicture}$ and $\begin{tikzpicture}[scale=0.25,baseline={([yshift=-.5ex]current bounding box.center)}]
\def\cz{5}
\def\wi{0.5}

\newcommand{\ci}[1]{	
	\fill[black] (#1) circle (\cz pt);
	\draw (#1) circle (\cz pt);
}

\coordinate (R) at (0,0);

\coordinate (r1) at (\wi,-1);
\coordinate (l1) at (-\wi,-1);

\coordinate (r11) at (\wi-\wi,-2);
\coordinate (r12) at (\wi+\wi,-2);

\draw (R) to (r1);
\draw (R) to (l1);

\draw (r1) to (r11);
\draw (r1) to (r12);

\ci{R}
\ci{r1}
\ci{l1}
\ci{r11}
\ci{r12}
\end{tikzpicture}$. A product (given by the disjoint union) of rooted trees will be called a (rooted) forest and by $\F$ we denote the $\Q$-algebra of forests generated by all trees. The unit of $\F$, given by the empty forest, will be denoted by $\I$. Since we just consider trees without plane structure the algebra $\F$ is commutative.  Due to the work of Connes and Kreimer (\cite{CK}) the space $\F$ has the structure of a Hopf algebra. To define the coproduct on $\F$ we first define the linear map $B_+$ on $\F$, which connects all roots of the trees in a forest to a new root. For example we have $B_+\left(\begin{tikzpicture}[scale=0.3,baseline={([yshift=-.5ex]current bounding box.center)}]
\def\cz{5}
\def\wi{0.5}

\newcommand{\ci}[1]{	
	\fill[black] (#1) circle (\cz pt);
	\draw (#1) circle (\cz pt);
}

\coordinate (R) at (0,0);
\coordinate (r1) at (\wi,-1);
\coordinate (l1) at (-\wi,-1);

\draw (R) to (l1);
\draw (R) to (r1);

\ci{R};
\ci{l1};
\ci{r1};
\end{tikzpicture} \,\,\begin{tikzpicture}[scale=0.3,baseline={([yshift=-.5ex]current bounding box.center)}]
\def\cz{5}
\def\wi{0.5}

\newcommand{\ci}[1]{	
	\fill[black] (#1) circle (\cz pt);
	\draw (#1) circle (\cz pt);
}

\coordinate (R) at (0,0);
\ci{R}
\end{tikzpicture}\right) = \begin{tikzpicture}[scale=0.25,baseline={([yshift=-.5ex]current bounding box.center)}]
\def\cz{5}
\def\wi{0.5}

\newcommand{\ci}[1]{	
	\fill[black] (#1) circle (\cz pt);
	\draw (#1) circle (\cz pt);
}

\coordinate (R) at (0,0);

\coordinate (r1) at (\wi,-1);
\coordinate (l1) at (-\wi,-1);

\coordinate (l11) at (-\wi-\wi,-2);
\coordinate (l12) at (-\wi+\wi,-2);

\draw (R) to (r1);
\draw (R) to (l1);

\draw (l1) to (l11);
\draw (l1) to (l12);

\ci{R}
\ci{r1}
\ci{l1}
\ci{l11}
\ci{l12}
\end{tikzpicture}$. Clearly for every tree $t \in \F$ there exists a unique forest $f_t \in \F$ with $t = B_+(f_t)$, which is just given by removing the root of $t$.
The coproduct on $\F$ can then be defined recursively for a tree $t \in \F$ by
\[ \Delta(t) = t \otimes \I + (\id \otimes B_{+}) \circ \Delta(f_t)\]
and for a forest $f=g h$ with $g,h \in \F$ multiplicatively by $\Delta(f) = \Delta(g)\Delta(h)$ and $\Delta(\I)= \I \otimes \I$. For example we have
\[\Delta(\begin{tikzpicture}[scale=0.3,baseline={([yshift=-.5ex]current bounding box.center)}]
\def\cz{5}
\def\wi{0.5}

\newcommand{\ci}[1]{	
	\fill[black] (#1) circle (\cz pt);
	\draw (#1) circle (\cz pt);
}

\coordinate (R) at (0,0);
\coordinate (r1) at (\wi,-1);
\coordinate (l1) at (-\wi,-1);

\draw (R) to (l1);
\draw (R) to (r1);

\ci{R};
\ci{l1};
\ci{r1};
\end{tikzpicture}) = \begin{tikzpicture}[scale=0.3,baseline={([yshift=-.5ex]current bounding box.center)}]
\def\cz{5}
\def\wi{0.5}

\newcommand{\ci}[1]{	
	\fill[black] (#1) circle (\cz pt);
	\draw (#1) circle (\cz pt);
}

\coordinate (R) at (0,0);
\coordinate (r1) at (\wi,-1);
\coordinate (l1) at (-\wi,-1);

\draw (R) to (l1);
\draw (R) to (r1);

\ci{R};
\ci{l1};
\ci{r1};
\end{tikzpicture} \otimes \I + \begin{tikzpicture}[scale=0.3,baseline={([yshift=-.5ex]current bounding box.center)}]
\def\cz{5}
\def\wi{0.5}

\newcommand{\ci}[1]{	
	\fill[black] (#1) circle (\cz pt);
	\draw (#1) circle (\cz pt);
}

\coordinate (R) at (0,0);
\ci{R}
\end{tikzpicture}\,\begin{tikzpicture}[scale=0.3,baseline={([yshift=-.5ex]current bounding box.center)}]
\def\cz{5}
\def\wi{0.5}

\newcommand{\ci}[1]{	
	\fill[black] (#1) circle (\cz pt);
	\draw (#1) circle (\cz pt);
}

\coordinate (R) at (0,0);
\ci{R}
\end{tikzpicture}  \otimes \begin{tikzpicture}[scale=0.3,baseline={([yshift=-.5ex]current bounding box.center)}]
\def\cz{5}
\def\wi{0.5}

\newcommand{\ci}[1]{	
	\fill[black] (#1) circle (\cz pt);
	\draw (#1) circle (\cz pt);
}

\coordinate (R) at (0,0);
\ci{R}
\end{tikzpicture} + 2 \,\begin{tikzpicture}[scale=0.3,baseline={([yshift=-.5ex]current bounding box.center)}]
\def\cz{5}
\def\wi{0.5}

\newcommand{\ci}[1]{	
	\fill[black] (#1) circle (\cz pt);
	\draw (#1) circle (\cz pt);
}

\coordinate (R) at (0,0);
\ci{R}
\end{tikzpicture}  \otimes \begin{tikzpicture}[scale=0.3,baseline={([yshift=-.5ex]current bounding box.center)}]
\def\cz{5}
\def\wi{0.5}

\newcommand{\ci}[1]{	
	\fill[black] (#1) circle (\cz pt);
	\draw (#1) circle (\cz pt);
}

\coordinate (R) at (0,0);
\coordinate (r1) at (0,-1);

\draw (R) to (r1);

\ci{R};
\ci{r1};

\end{tikzpicture}+ \I\otimes \begin{tikzpicture}[scale=0.3,baseline={([yshift=-.5ex]current bounding box.center)}]
\def\cz{5}
\def\wi{0.5}

\newcommand{\ci}[1]{	
	\fill[black] (#1) circle (\cz pt);
	\draw (#1) circle (\cz pt);
}

\coordinate (R) at (0,0);
\coordinate (r1) at (\wi,-1);
\coordinate (l1) at (-\wi,-1);

\draw (R) to (l1);
\draw (R) to (r1);

\ci{R};
\ci{l1};
\ci{r1};
\end{tikzpicture}\,.\]
 In \cite{T} the second named author uses the coproduct $\Delta$ to assign to a forest $f\in \F$ a $\Q$-linear map on the space $\h = \Q\langle x,y \rangle$, called a rooted tree map, by the following:
\begin{definition}\label{def:rtm} For any non-empty forest $f \in \F$, we define a $\Q$-linear map on $\h$, also denoted by $f$, recursively: For a word $w \in \h$ and a letter $u \in \{x,y\}$ we set
\[ f(w u):= M(\Delta(f)(w \otimes u))\,, \]
where $M(w_1 \otimes w_2) = w_1 w_2$ denotes the multiplication on $\h$. This reduces the calculation to $f(u)$ for a letter $u \in \{x,y\}$, which is defined by the following:
\begin{enumerate}[i)]
\item If $f=\begin{tikzpicture}[scale=0.3,baseline={([yshift=-.5ex]current bounding box.center)}]
\def\cz{5}
\def\wi{0.5}

\newcommand{\ci}[1]{	
	\fill[black] (#1) circle (\cz pt);
	\draw (#1) circle (\cz pt);
}

\coordinate (R) at (0,0);
\ci{R}
\end{tikzpicture}\,$, then $f(x) := xy$ and $f(y) := -xy$.
\item For a tree $t = B_+(f)$ we set $t(u) := R_y R_{x+2y}R_y^{-1} f(u)$,
where $R_v$ is the linear map given by $R_v(w)=wv$ ($v,w \in \h$).
\item If $f = gh$ is a forest with $g,h \neq \I$, then $f(u):=g(h(u))$.
\end{enumerate}
The rooted tree map of the empty forest $\I$ is given by the identity.
\end{definition}
In \cite[Theorem 1.1]{T} it is proven that this definition, in particular ii), is well-defined.  

Let $\h^0 = \Q + x\h y \subset \h$ be the subspace of admissible words and define the $\Q$-linear map $Z: \h^0 \rightarrow \R$ on a monomial $w=z_{k_1}\dots z_{k_r}$ with $z_k=x^{k-1}y \in \h$ ($k\geq 1$) by
\begin{equation} \label{eq:defmzv}
Z(w) = \zeta(k_1,\dots,k_r) = \sum_{m_1 > \dots > m_r > 0} \frac{1}{m_1^{k_1} \dots m_r^{k_r}} \,.
\end{equation}
The $\zeta(k_1,\dots,k_r)$ are called multiple zeta values and one particular interests in these real numbers is to understand their $\Q$-linear relations. Note that $z_{k_1}\dots z_{k_r} \in \h^0$ implies that $k_1 \geq 2$, $k_2,\dots,k_r\geq 1$, which ensures that the sum on the right of \eqref{eq:defmzv} converges. One application of the rooted tree maps is the following result in \cite{T}.
\begin{theorem}(\cite[Theorem 1.3]{T})\label{thm:rtmrelation} For any non-empty forest $f \in \F$ we have $$f(\h^0) \subset \ker Z.$$
\end{theorem}

\begin{ex}
For the tree $f=\begin{tikzpicture}[scale=0.3,baseline={([yshift=-.5ex]current bounding box.center)}]
\def\cz{5}
\def\wi{0.5}

\newcommand{\ci}[1]{	
	\fill[black] (#1) circle (\cz pt);
	\draw (#1) circle (\cz pt);
}

\coordinate (R) at (0,0);
\coordinate (r1) at (0,-1);

\draw (R) to (r1);

\ci{R};
\ci{r1};

\end{tikzpicture}$ and the word $w=xy$ we obtain for $f(w)$
\begin{align*}
\begin{tikzpicture}[scale=0.3,baseline={([yshift=-.5ex]current bounding box.center)}]
\def\cz{5}
\def\wi{0.5}

\newcommand{\ci}[1]{	
	\fill[black] (#1) circle (\cz pt);
	\draw (#1) circle (\cz pt);
}

\coordinate (R) at (0,0);
\coordinate (r1) at (0,-1);

\draw (R) to (r1);

\ci{R};
\ci{r1};

\end{tikzpicture}(xy) = M(\Delta(\begin{tikzpicture}[scale=0.3,baseline={([yshift=-.5ex]current bounding box.center)}]
\def\cz{5}
\def\wi{0.5}

\newcommand{\ci}[1]{	
	\fill[black] (#1) circle (\cz pt);
	\draw (#1) circle (\cz pt);
}

\coordinate (R) at (0,0);
\coordinate (r1) at (0,-1);

\draw (R) to (r1);

\ci{R};
\ci{r1};

\end{tikzpicture})(x \otimes y))= M(\begin{tikzpicture}[scale=0.3,baseline={([yshift=-.5ex]current bounding box.center)}]
\def\cz{5}
\def\wi{0.5}

\newcommand{\ci}[1]{	
	\fill[black] (#1) circle (\cz pt);
	\draw (#1) circle (\cz pt);
}

\coordinate (R) at (0,0);
\coordinate (r1) at (0,-1);

\draw (R) to (r1);

\ci{R};
\ci{r1};

\end{tikzpicture}(x) \otimes y + \begin{tikzpicture}[scale=0.3,baseline={([yshift=-.5ex]current bounding box.center)}]
\def\cz{5}
\def\wi{0.5}

\newcommand{\ci}[1]{	
	\fill[black] (#1) circle (\cz pt);
	\draw (#1) circle (\cz pt);
}

\coordinate (R) at (0,0);
\ci{R}
\end{tikzpicture}(x) \otimes \begin{tikzpicture}[scale=0.3,baseline={([yshift=-.5ex]current bounding box.center)}]
\def\cz{5}
\def\wi{0.5}

\newcommand{\ci}[1]{	
	\fill[black] (#1) circle (\cz pt);
	\draw (#1) circle (\cz pt);
}

\coordinate (R) at (0,0);
\ci{R}
\end{tikzpicture}(y) +x \otimes \begin{tikzpicture}[scale=0.3,baseline={([yshift=-.5ex]current bounding box.center)}]
\def\cz{5}
\def\wi{0.5}

\newcommand{\ci}[1]{	
	\fill[black] (#1) circle (\cz pt);
	\draw (#1) circle (\cz pt);
}

\coordinate (R) at (0,0);
\coordinate (r1) at (0,-1);

\draw (R) to (r1);

\ci{R};
\ci{r1};

\end{tikzpicture}(y)).
\end{align*}
Together with  $\begin{tikzpicture}[scale=0.3,baseline={([yshift=-.5ex]current bounding box.center)}]
\def\cz{5}
\def\wi{0.5}

\newcommand{\ci}[1]{	
	\fill[black] (#1) circle (\cz pt);
	\draw (#1) circle (\cz pt);
}

\coordinate (R) at (0,0);
\ci{R}
\end{tikzpicture}(x)=xy$ and $\begin{tikzpicture}[scale=0.3,baseline={([yshift=-.5ex]current bounding box.center)}]
\def\cz{5}
\def\wi{0.5}

\newcommand{\ci}[1]{	
	\fill[black] (#1) circle (\cz pt);
	\draw (#1) circle (\cz pt);
}

\coordinate (R) at (0,0);
\coordinate (r1) at (0,-1);

\draw (R) to (r1);

\ci{R};
\ci{r1};

\end{tikzpicture}(x)=R_y R_{x+2y} R_y^{-1} \begin{tikzpicture}[scale=0.3,baseline={([yshift=-.5ex]current bounding box.center)}]
\def\cz{5}
\def\wi{0.5}

\newcommand{\ci}[1]{	
	\fill[black] (#1) circle (\cz pt);
	\draw (#1) circle (\cz pt);
}

\coordinate (R) at (0,0);
\ci{R}
\end{tikzpicture}(x)=x(x+2y)y$ we get
\[\begin{tikzpicture}[scale=0.3,baseline={([yshift=-.5ex]current bounding box.center)}]
\def\cz{5}
\def\wi{0.5}

\newcommand{\ci}[1]{	
	\fill[black] (#1) circle (\cz pt);
	\draw (#1) circle (\cz pt);
}

\coordinate (R) at (0,0);
\coordinate (r1) at (0,-1);

\draw (R) to (r1);

\ci{R};
\ci{r1};

\end{tikzpicture}(xy) = 2xyyy - xyxy-xxxy-xxyy = 2z_2 z_1 z_1 - z_2 z_2 - z_4 - z_3 z_1 \,,\]
which by Theorem \ref{thm:rtmrelation} gives the linear relation $2\zeta(2,1,1)=\zeta(4)+\zeta(2,2)+\zeta(3,1)$.
\end{ex}
 In \cite{IKZ} the authors define the derivation $\partial_n$ on $\h$ by $\partial_n(x) = x(x+y)^{n-1}y$ and $\partial_n(y)=-x(x+y)^{n-1}y$. Also $\partial_n$ gives linear relation between multiple zeta values, since for all $n\geq 1$ we have $\partial_n(\h^0) \subset  \ker Z$ (\cite[Corollary 6]{IKZ}). These relations are known as the derivation relations for multiple zeta values. 
 
 The main result of this note is to show that the derivation relations are just a special case of Theorem \ref{thm:rtmrelation}, by giving an explicit description of $\partial_n$ in terms of rooted tree maps. For this just trees without any branches are needed, i.e. we will consider for $m \geq 1$ the ladder trees
 \[ \lambda_m = \begin{tikzpicture}[scale=0.3,baseline={([yshift=-.5ex]current bounding box.center)}]
\def\cz{5}
\def\wi{0.5}

\newcommand{\ci}[1]{	
	\fill[black] (#1) circle (\cz pt);
	\draw (#1) circle (\cz pt);
}

\coordinate (R) at (0,0);
\coordinate (r1) at (0,-1);
\coordinate (r11) at (0,-1.4);

\coordinate (r21) at (0,-2.6);
\coordinate (r2) at (0,-3);
\coordinate (r3) at (0,-4);

\draw (R) to (r1);

\draw (r1) to (r11);
\draw[densely dotted] (r11) to (r21);
\draw (r21) to (r2);
\draw (r2) to (r3);

\ci{R};
\ci{r1};
\ci{r2};
\ci{r3};

\draw[decoration={brace,mirror,raise=5pt},decorate]
  (r3) -- node[right=6pt] {$\footnotesize{m}$} (R);

\end{tikzpicture} \]
 
 and set $\lambda_0 = \I$.  With this the main result of this work is the following.
\begin{theorem}\label{thm:mainthm}For all $n \geq 1$ the derivation $\partial_n$ is given by
 \begin{equation}\label{eq:partialn}
  \partial_n = \frac{n}{2^n-1} \sum_{d=1}^n \frac{(-1)^{d+1}}{d} \sum_{\substack{m_1+\dots+m_d=n\\m_1,\dots,m_d \geq 1}}  \lambda_{m_1} \dots \lambda_{m_d}  \,.
  \end{equation}
\end{theorem}
Recall that by $\lambda_{m_1} \dots \lambda_{m_d}$ we denote the rooted tree maps corresponding to the forest of ladder trees with heights $m_1,\dots,m_d$. By Definition \ref{def:rtm} iii) we have $\lambda_{m_1} \lambda_{m_2} = \lambda_{m_2} \lambda_{m_1}$, so we get for the first few values of $n$  ($w \in \h$)
\begin{align*}\partial_1(w) = \begin{tikzpicture}[scale=0.3,baseline={([yshift=-.5ex]current bounding box.center)}]
\def\cz{5}
\def\wi{0.5}

\newcommand{\ci}[1]{	
	\fill[black] (#1) circle (\cz pt);
	\draw (#1) circle (\cz pt);
}

\coordinate (R) at (0,0);
\ci{R}
\end{tikzpicture}(w)\,,\quad \partial_2(w) = \frac{2}{3}  \,\begin{tikzpicture}[scale=0.3,baseline={([yshift=-.5ex]current bounding box.center)}]
\def\cz{5}
\def\wi{0.5}

\newcommand{\ci}[1]{	
	\fill[black] (#1) circle (\cz pt);
	\draw (#1) circle (\cz pt);
}

\coordinate (R) at (0,0);
\coordinate (r1) at (0,-1);

\draw (R) to (r1);

\ci{R};
\ci{r1};

\end{tikzpicture}(w) - \frac{1}{3}\,\begin{tikzpicture}[scale=0.3,baseline={([yshift=-.5ex]current bounding box.center)}]
\def\cz{5}
\def\wi{0.5}

\newcommand{\ci}[1]{	
	\fill[black] (#1) circle (\cz pt);
	\draw (#1) circle (\cz pt);
}

\coordinate (R) at (0,0);
\ci{R}
\end{tikzpicture}\,\begin{tikzpicture}[scale=0.3,baseline={([yshift=-.5ex]current bounding box.center)}]
\def\cz{5}
\def\wi{0.5}

\newcommand{\ci}[1]{	
	\fill[black] (#1) circle (\cz pt);
	\draw (#1) circle (\cz pt);
}

\coordinate (R) at (0,0);
\ci{R}
\end{tikzpicture}(w) \,,\quad
\partial_3(w) = \frac{3}{7} \,\begin{tikzpicture}[scale=0.3,baseline={([yshift=-.5ex]current bounding box.center)}]
\def\cz{5}
\def\wi{0.5}

\newcommand{\ci}[1]{	
	\fill[black] (#1) circle (\cz pt);
	\draw (#1) circle (\cz pt);
}

\coordinate (R) at (0,0);
\coordinate (r1) at (0,-1);
\coordinate (r2) at (0,-2);

\draw (R) to (r1);
\draw (r1) to (r2);
\ci{R};
\ci{r1};
\ci{r2};
\end{tikzpicture}(w) -  \frac{3}{7} \,\begin{tikzpicture}[scale=0.3,baseline={([yshift=-.5ex]current bounding box.center)}]
\def\cz{5}
\def\wi{0.5}

\newcommand{\ci}[1]{	
	\fill[black] (#1) circle (\cz pt);
	\draw (#1) circle (\cz pt);
}

\coordinate (R) at (0,0);
\coordinate (r1) at (0,-1);

\draw (R) to (r1);

\ci{R};
\ci{r1};

\end{tikzpicture}\,\begin{tikzpicture}[scale=0.3,baseline={([yshift=-.5ex]current bounding box.center)}]
\def\cz{5}
\def\wi{0.5}

\newcommand{\ci}[1]{	
	\fill[black] (#1) circle (\cz pt);
	\draw (#1) circle (\cz pt);
}

\coordinate (R) at (0,0);
\ci{R}
\end{tikzpicture}(w)+\frac{1}{7}\begin{tikzpicture}[scale=0.3,baseline={([yshift=-.5ex]current bounding box.center)}]
\def\cz{5}
\def\wi{0.5}

\newcommand{\ci}[1]{	
	\fill[black] (#1) circle (\cz pt);
	\draw (#1) circle (\cz pt);
}

\coordinate (R) at (0,0);
\ci{R}
\end{tikzpicture}\,\begin{tikzpicture}[scale=0.3,baseline={([yshift=-.5ex]current bounding box.center)}]
\def\cz{5}
\def\wi{0.5}

\newcommand{\ci}[1]{	
	\fill[black] (#1) circle (\cz pt);
	\draw (#1) circle (\cz pt);
}

\coordinate (R) at (0,0);
\ci{R}
\end{tikzpicture}\,\begin{tikzpicture}[scale=0.3,baseline={([yshift=-.5ex]current bounding box.center)}]
\def\cz{5}
\def\wi{0.5}

\newcommand{\ci}[1]{	
	\fill[black] (#1) circle (\cz pt);
	\draw (#1) circle (\cz pt);
}

\coordinate (R) at (0,0);
\ci{R}
\end{tikzpicture}(w)\,.
\end{align*}
Rewriting \eqref{eq:partialn} gives the following recursive formula for the rooted tree map $\lambda_n$
\begin{align*}
n \lambda_n = \sum_{j=1}^n \left(2^j -1\right) \lambda_{n-j} \partial_j \,.
\end{align*}

\subsection*{Acknowledgment}
The authors would like to thank the Max-Planck-Institut f\"ur Mathematik in Bonn for hospitality and support. The second author was also supported by Kyoto Sangyo University Research Grants.

\section{Dynkin operator for ladders}

By $S$ we denote the antipode of $\F$ and by $Y$ we denote the grading operator on $\F$ given by $Y(f)= \deg(f) f$. Here $\deg(f)$ is the degree of $f \in \F$ given by the number of vertices. To prove Theorem \ref{thm:mainthm} we define for all $f \in \F$ the Dynkin operator $D$ by
\[ D(f) := (m\circ(S \otimes Y)\circ\Delta)(f) \,,\]
where $m$ denotes the multiplication on $\F$.
First we give an explicit expression for $D(\lambda_n)$, then show that the associated rooted tree map $D(\lambda_n)$ is a derivation on $\h$ and finally prove $D(\lambda_n)(x)=(2^n-1) \partial_n(x)$. From this we obtain for all words $w\in \h$ the equality $D(\lambda_n)(w)=(2^n-1) \partial_n(w)$ and the identity in Theorem  \ref{thm:mainthm}.

\begin{lemma}\label{lem:identities}
For all $n\geq 1$ we have 
\begin{align}\label{eq:deltalambdan}
\Delta(\lambda_n) &= \sum_{j=0}^n \lambda_j \otimes \lambda_{n-j}\,,\\\label{eq:slambdan}
S(\lambda_n) &= \sum_{d=1}^{n} (-1)^d \sum_{\substack{m_1+\dots+m_d=n\\m_1,\dots,m_d \geq 1}} \lambda_{m_1} \dots \lambda_{m_d}\,,\\ \label{eq:dident}
D(\lambda_n) &= n\sum_{d=1}^n \frac{(-1)^{d+1}}{d} \sum_{\substack{m_1+\dots+m_d=n\\m_1,\dots,m_d \geq 1}} \lambda_{m_1} \dots \lambda_{m_d}\,.
\end{align}
\end{lemma}
\begin{proof}
The first formula follows inductively by definition of the coproduct, since $\Delta(\lambda_n) = \lambda_n \otimes \I + (\id \otimes B_{+}) \circ \Delta(\lambda_{n-1})$. The condition for $S$ being an antipode means in particular that $(m \circ (S \otimes \id) \circ \Delta)(\lambda_n) =0$. With $\eqref{eq:deltalambdan}$ this gives the condition $S(\lambda_n)=-\sum_{j=0}^{n-1} S(\lambda_j) \lambda_{n-j}$, from which \eqref{eq:slambdan} follows by induction on $n$. To prove the third statement we get by the first and second equation
\begin{align*}
D(\lambda_n) &= (m\circ(S \otimes Y)\circ\Delta)(\lambda_n) = (m\circ(S \otimes Y))\left(\sum_{j=0}^n \lambda_j \otimes \lambda_{n-j}\right)\\ &=\sum_{j=0}^{n}(n-j)S(\lambda_j)\lambda_{n-j} = \sum_{d=1}^n (-1)^{d+1} \sum_{\substack{m_1+\dots+m_d=n\\m_1,\dots,m_d \geq 1}} m_d \lambda_{m_1} \dots \lambda_{m_d}\,.
\end{align*}
Since the last sum is symmetric in the $m_\ast$, we can replace the factor $m_d$ by any other $m_j$ with $1 \leq j \leq d$. Averaging over all $d$ choices for $j$ gives \eqref{eq:dident}.
\end{proof}

\begin{lemma}\label{lem:derivation}
For all $n\geq 1$ the rooted tree map $D(\lambda_n)$ is a derivation on $\h$.
\end{lemma}
\begin{proof}
By Proposition 3.8 in \cite{T} we have for any words $v,w \in \h$ and any forest $f \in \F$, $f(vw) =  M(\Delta(f)(v\otimes w))$. In the definition of rooted tree maps this is just given for the case $w$ being a letter, but it can be proven inductively on the number of letters of the word $vw$.
We need to show that $(m\circ(S \otimes Y)\circ\Delta)(\lambda_n)$ is primitive, since for a primitive element $f \in \F$, we get
\[f(vw) =  M(\Delta(f)(v\otimes w)) = M(f(v) \otimes w + v \otimes f(w)) = f(v)w + vf(w) \,.\]
In \cite[Theorem 5]{PR} it is shown that the Dynkin operator $m\circ(S \otimes Y)\circ\Delta$ applied to a cocommutative element gives an primitive element. By \eqref{eq:deltalambdan} the $\lambda_n$ are cocommutative from which the statement follows. 
\end{proof}

\begin{lemma} \label{lem:xequ}
For all $n \geq 1$ we have $D(\lambda_n)(x)=(2^n-1) \partial_n(x)$.
\end{lemma}
\begin{proof}
Let $\h[[u]]$ be the formal power series ring over $\h$ with indeterminate $u$ and let $\Delta_{u}$  be the automorphism of $\h[[u]]$ whose images on the generators is given by $\Delta_u(u)=u$,  $\Delta_u(x) = x (1+yu)^{-1}$ and $\Delta_u(y) = y + x(1+yu)^{-1}$.
By direct calculation one checks that 
\begin{equation}\label{eq:deltaident}
\left( \Delta_{-2u} \circ \Delta^{-1}_{-u}\right)(x) = x +  x \frac{u}{1-(x+2y)u} y \,.
\end{equation} 
In \cite[Theorem 4]{IKZ} it is proven that \[\Delta_u = \exp\left( \sum_{n=1}^\infty (-1)^{n}\frac{\partial_n}{n} u^n \right)\,,\]
which together with \eqref{eq:deltaident} gives
\begin{equation}\label{eq:expident}
\exp\left( \sum_{n=1}^\infty (2^{n}-1)\frac{\partial_n(x)}{n} u^n \right)= x +  x \frac{u}{1-(x+2y)u} y\,.
\end{equation}
Now define on $\h[[u]]$ the automorphism $\Lambda_u = \sum_{n\geq 0} \lambda_n u^n$ and calculate
\begin{align*}
\log(\Lambda_u) &= \sum_{d\geq 1} \frac{(-1)^{d+1}}{d} \left( \sum_{n\geq 1} \lambda_n u^n\right)^d =\sum_{n\geq 1} \sum_{d\geq 1} \frac{(-1)^{d+1}}{d}\sum_{\substack{m_1+\dots+m_d=n\\m_1,\dots,m_d \geq 1}} \lambda_{m_1} \dots \lambda_{m_d} u^n\,.
\end{align*}
By Lemma \ref{lem:identities}  this gives $\log(\Lambda_u) = \sum_{n\geq 1} \frac{D(\lambda_n)}{n} u^n$. 
The definition of rooted tree maps implies $\lambda_n(x) = R_y R_{x+2y} R_y^{-1} \lambda_{n-1}(x) = \dots = R_y R_{x+2y}^{n-1} R_y^{-1} \lambda_{1}(x) = x (x+2y)^{n-1} y$. Therefore the image of $x$ under $\Lambda_u$ is given by
\[\Lambda_u(x) = \sum_{n\geq 0} \lambda_n(x) u^n =x +
 \sum_{n\geq 1} x(x+2y)^{n-1}y u^n= x +  x \frac{u}{1-(x+2y)u} y\,. \]
 Since this equals $\exp(\sum_{n\geq 1} \frac{D(\lambda_n)(x)}{n} u^n)$ we obtain the desired identity by \eqref{eq:expident}.
\end{proof}

\begin{proof}[Proof of Theorem \ref{thm:mainthm}] By Lemma \ref{lem:derivation} the rooted tree map $D(\lambda_n)$ is a derivation on $\h$, which satisfies (as every rooted tree map) $D(\lambda_n)(y) = -D(\lambda_n)(x)$. By Lemma \ref{lem:xequ} and because of $\partial_n(y)=-\partial_n(x)$ the derivations  $D(\lambda_n)$ and $(2^n-1) \partial_n$ are the same on the generators of $\h$ and hence they are equal. The explicit formula for $\partial_n$ in \eqref{eq:partialn} now follows from \eqref{eq:dident} in Lemma \ref{lem:identities}.\end{proof}

\begin{remark}
As seen in Lemma \ref{lem:derivation} the Dynkin operator of any cocommutative element in $\F$ gives a derivation on $\h$. A natural question therefore is, if there are other cocommutative elements in $\F$, which gives rise to derivations on $\h$. For example the element $\lambda_1^n$ is cocommutative, since $\Delta(\lambda_1^n)=\sum_{j=0}^n \binom{n}{j} \lambda_1^j \otimes \lambda_1^{n-j}$. But with $S(\lambda_1^n)=(-1)^n \lambda_1^n$ one checks that $D(\lambda_1^n)=\sum_{j=1}^n \binom{n}{j}(-1)^{j}(n-j) \lambda_1^n =0$ for $n\geq 2$, which does not give an interesting example of a derivation.
\end{remark}


\end{document}